\documentclass{amsart}

\usepackage[cmtip,all]{xy}
\usepackage{tikz}
\usetikzlibrary{matrix,cd}
\usepackage{mathtools}
\usepackage{mathabx}
\usepackage{mathrsfs}
\usepackage{float}
\usepackage{hyperref}
\usepackage{cleveref}
\usepackage{tabstackengine}

\usepackage{bbm}
\usepackage{ytableau}
\newtheorem{theorem}{Theorem}[section]
\newtheorem{lemma}[theorem]{Lemma}

\newtheorem{corollary}[theorem]{Corollary}

\newtheorem*{theorem*}{Theorem}

\newtheorem*{theoremSec}{Theorem \ref{sec}}

\theoremstyle{definition}
\newtheorem{definition}[theorem]{Definition}
\newtheorem{example}[theorem]{Example}

\theoremstyle{remark}
\newtheorem{remark}[theorem]{Remark}

\numberwithin{equation}{section}

\newcommand{\N}{\mathbb{N}}
\newcommand{\Z}{\mathbb{Z}}
\newcommand{\PP}{\mathbb{P}}
\newcommand{\BC}{\mathcal{B}}

\newcommand{\Tstrut}{\rule{0pt}{3ex}}
\newcommand{\Bstrut}{\rule[-1.6ex]{0pt}{0pt}}

\newcommand\A[1]{{A_{#1}(\lambda,\mu)}}
\newcommand\B[1]{B_{#1}(\lambda,\mu)}
\newcommand\g[2]{g_{#1}(\lambda,\mu;{#2})}
\newcommand\h[2]{h_{#1}(\lambda,\mu;{#2})}

\newcommand{\ugly}[1]{{[#1]}}

\newcommand{\order}{\succeq}
\newcommand{\nonzero}{\mu\order\lambda}
\newcommand{\aij}[2]{(\lambda^{\ugly {#1}}, \mu^{\ugly {#2}})}

\newcommand{\genedifflow}[2]{I_{{#1}_{#2}}}

\newcommand{\ev}{\mathrm{ev}}

\newcommand{\bigzero}{\mbox{\normalfont\Large\bfseries 0}}

\makeatletter
\newcommand{\exterior}[1]{\mathop{\mathpalette\exterior@{#1}}}
\newcommand{\exterior@}[2]{
  \raisebox{\depth}{
  \fontsize{\sf@size}{0}
  \m@th
  $\ifx#1\displaystyle\textstyle\else#1\fi\bigwedge$}
  ^{\mspace{-2mu}#2}
  \kern-\scriptspace
}
\makeatother

\begin{document}

\title{A Uniform Identification of Stable Sheaf Cohomology}

\author{Luca Fiorindo} 
\address{Dipartimento di Matematica, Dipartimento di Eccellenza 2023-2027, Università di Genova, via Dodecaneso 35, 16146 Genova, Italy}
\email{luca.fiorindo@edu.unige.it}
\thanks{The first and the fourth authors are partly supported by the MIUR Excellence Department Project awarded to Dipartimento di Matematica, Università di Genova, CUP\, D33C23001110001, and are members of GNSAGA (INdAM)}
\urladdr{\url{https://orcid.org/0000-0002-6435-0128}}

\author{Ethan Reed}
\address{Department of Mathematics, University of Notre Dame, 255 Hurley, Notre Dame, IN 46556}
\email{ereed4@nd.edu}
\thanks{The second author was partially supported by National Science Foundation Grant DMS-2302341.}
\urladdr{\url{https://orcid.org/0009-0003-6620-8664}}

\author{Shahriyar Roshan Zamir}
\address{Department of Mathematics, University of Nebraska-Lincoln, 1144 T St Lincoln, NE 68588}
\email{sroshanzamir2@huskers.unl.edu}
\thanks{The third author received support from National Science Foundation Grant DMS-2101225.}

\author{Hongmiao Yu}
\address{Dipartimento di Matematica, Dipartimento di Eccellenza 2023-2027, Università di Genova, via Dodecaneso 35, 16146 Genova, Italy}
\email{yu@dima.unige.it}
\thanks{The fourth author is supported by PRIN 2020355B8Y ``Squarefree Gr\"obner degenerations, special varieties and related topics''.}


\begin{abstract}
This paper considers generalizations of certain arithmetic complexes appearing in the work of Raicu and VandeBogert in connection with the study of stable sheaf cohomology on flag varieties. Defined over the ring of integer valued polynomials, we prove an isomorphism of these complexes as conjectured by Gao, Raicu, and VandeBogert. In particular, this shows that a previously made identification between the stable sheaf cohomology of hook and two column partition Schur functors applied to the cotangent sheaf of projective space can be made to be uniform with respect to these complexes. These results are extended to the projective space defined over the integers.
\end{abstract}

\maketitle

\section{Introduction}
A central problem connecting representation theory and algebraic geometry is the calculation of sheaf cohomology of line bundles on full flag varieties over an algebraically closed field ${\bf k}$. Recall that for every natural number $n$, the flag variety $Fl_n$ is the variety parameterizing complete flags of a vector space $V$ of dimension $n$, where a complete flag of $V$ is a filtration $0\subset V_1\subset\dots\subset V_{n-1}\subset V={\bf k}^n$ with $\dim_{{\bf k}} V_i=i$. In characteristic zero, the Borel-Weil-Bott theorem \cite{borelweil}, \cite{bott} computes the sheaf cohomology groups and shows line bundles have at most one non-zero cohomology group which is an irreducible representation of the general linear group. However, the sheaf cohomology groups can be much more complicated in positive characteristic. 
For instance, Raicu and VandeBogert \cite{raicukeller} showed that the number of irreducible factors in a filtration of the cohomology groups cannot be bounded by a polynomial in terms of the cohomological degree, dimension of the flag variety, and the combinatorial data defining the line bundle. Nonetheless, some notable results in positive characteristic are: the Kempf vanishing theorem \cite{kempf}, \cite{Kempf-LinearSystems}, \cite{Kempf-Flagmanifolds}, \cite{haboush}, \cite{Ander-Frob}, a characterization of non-vanishing of the first cohomology group by Andersen \cite{Andersen}, and a full description when the underlying vector space is $3$-dimensional by Donkin \cite{Donkin}.

The most relevant result in positive characteristic for this paper is a recent stability result of Raicu and VandeBogert \cite{raicukeller}. They show that for fixed input data, the sheaf cohomology groups can be described uniformly using polynomial functors as long as the dimension of $V$ is large enough. This motivates the terminology \textit{stable sheaf cohomology}. A related statement of \cite{raicukeller} is the following: for the cotangent sheaf $\Omega$ on projective space $\mathbb{P}^n({\bf k})$, and a fixed polynomial functor $\mathcal{P}$, the sheaf cohomology groups $H^j(\mathbb{P}^n({\bf k}), \mathcal{P}(\Omega))$ have trivial $\text{GL}_n({\bf k})$-action and are independent of $n$, for $n$ sufficiently large. In the special case when $\mathcal{P}=\mathbb{S}_{\lambda}$ is a Schur functor, the sheaf cohomology groups $H^j(\mathbb{P}^n({\bf k}), \mathbb{S}_{\lambda}(\Omega))$ coincide with the sheaf cohomology groups of certain line bundles on the flag variety $Fl_n$. In the case when $\lambda$ is a ribbon or two column Schur functor, the former of which enjoys nice combinatorial properties (\cite{lascouxpragacz}, \cite{Reiner}), Raicu and VandeBogert construct explicit universal resolutions of the aforementioned functors, the existence of which stems from \cite{AB88}, to derive \textit{arithmetic complexes}, denoted $C^{\bf k}_{\bullet}(w_0,\ldots,w_n)$, and employ them to show an identification for stable sheaf cohomology of two column and hook Schur functors.

The complexes $C_{\bullet}(\underline{w})$ are our main topic of study; they are defined in \Cref{section2} over $R = \mathbb{Z}\langle \binom{x}{n}\rangle$, the ring of integer valued polynomials, as opposed to their original definition over ${\bf k}$. This definition allows a formulation of a conjecture by Gao, Raicu, and VandeBogert \cite[Conjecture 2.3]{raicukellerZhao}: the identification of sheaf cohomology of two column and hook Schur functors is uniform with respect to varying the length of the first column of both. A proof of this conjecture is the crux of \Cref{newresults} and is now stated as a theorem.

\begin{theorem}\label{main}
For every integer $d\ge 1$, there exists an isomorphism of complexes 
    \[C_\bullet(x,1^d)\cong C_\bullet(-x-2d,1^d).\]
\end{theorem}
Here $1^d$ records a repeated entry of $1$ $d$-times. Further, using the techniques in \cite{raicukeller} one immediately obtains analogous stability results for $\mathbb{S}_{\lambda / \mu}\Omega$ for the projective space over the integers where $\mathbb{S}_{\lambda / \mu}$ denotes a skew Schur functor. These functors over arbitrary commutative rings are discussed in \cite{akinbuchsbaumweyman}. For ribbons or two column partitions, the stable cohomology groups are given by the homology of complexes $C_{\bullet}^{\mathbb{Z}}(\underline{w})$, now interpreted as complexes of free abelian groups. The discussion above is made precise in \Cref{applications} and motivates the next result which relies on \Cref{main}.

\begin{theorem}\label{sec}
     Let $\Omega$ be the cotangent sheaf over $\mathbb{P}^{n}(\mathbb{Z})$ and $\mathbb{S}_{\lambda/ \mu}$ be a skew Schur functor. The following is true for all $i$.
    \begin{enumerate}
    \item[$i)$] $H^i(\mathbb{P}^n(\Z), \mathbb{S}_{\lambda/\mu} \Omega)$ is independent of $n$ for $n \gg 0$, and is denoted by\\ $H^i_{st}(\mathbb{S}_{\lambda/\mu} \Omega)$.
    \item[$ii)$] If $\mathbb{S}_{\lambda/\mu}$ is a ribbon Schur functor corresponding to a $(w_0, ... , w_d)$, then 
    \[H^i_{st}(\mathbb{S}_{\lambda/\mu} \Omega) = H_{|w|-i}(C_{\bullet}(\underline{w})\otimes_R \mathbb{Z}).\]
    \item[$iii)$] If $\lambda$ is a two column partition, i.e. the conjugate partition has the form $\lambda' = (m, d)$ for integers $m,d\geq 0$, then 
    \[H_{st}^i(\mathbb{S}_{\lambda} \Omega) = H_{d+m-i}(\widecheck{C_{\bullet}}(-m-d-1,1^d)[-d]\otimes_R \mathbb{Z}),\]
    where $\widecheck{C_{\bullet}}(-m-d-1,1^d)$ denotes the dual of the complex $C_{\bullet}(-m-d-1,1^d)$.
    \item[$iv)$] Consider $\lambda$ as in $iii)$ and let $d \geq 1$, then
    \[H_{st}^i(\mathbb{S}_{\lambda} \Omega) = H_{st}^{2m+2-i}(\mathbb{S}_{(m-d+1, 1^d)}\Omega).\]
    \end{enumerate}
\end{theorem}

\section{Arithmetic Complexes} \label{section2}

Let $R= \mathbb{Z}\langle\binom{x}{n}\rangle$. In this paper, our main objects of study are $C_{\bullet}(\underline{w})$ with $\underline{w}=(w_0,\dots,w_d)\in R\times \N^{d-1}$. The definition in \cite[Section 5]{raicukeller} can be recovered by choosing any integer evaluation map $\ev:R\longrightarrow\Z$. One also gets a ring map $R\longrightarrow\bf k$ by post-composing the natural morphism $\Z\longrightarrow {\bf k}$, then
$$C_{\bullet}(w_0, w_1, ... , w_d)\otimes_R {\bf k} \cong C_{\bullet}^{\bf k}(\ev(w_0), w_1, ... ,w_d).$$
Similarly, one can define $C_{\bullet}^{\mathbb{Z}}(\underline{w})$ by using any evaluation map to $\mathbb{Z}$ to tensor over $\mathbb{Z}$.

For a combinatorial description, consider the undirected, weighted path graph of $d+1$ vertices and $d$ edges:
\begin{figure}[H]
\centerline{
\xymatrix@R=0pt{
 \overset{w_0}{\bullet}\ar@{-}[r]&\overset{w_1}{\bullet}\ar@{-}[r]&\overset{w_2}{\bullet}\ar@{..}[r]&\overset{w_{d-1}}{\bullet}\ar@{-}[r]&\overset{w_d}{\bullet}}}
 \ \\
 \caption{Undirected Weighted Path Graph with Vertex Weights $\underline{w}$}
 \label{UWPG}
\end{figure}
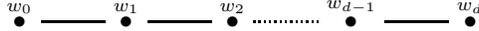
\noindent where $w_i$'s denote weights on the vertices. Given $d,t\in\N$ such that $1\le t\leq d+1$ for any weak composition of $d-t+1$, that is non-negative integers $(\lambda_t, \ldots, \lambda_1)$ such that $\sum_{i=1}^{t} \lambda_i = d-t+1$, there is a  unique decomposition of the above path graph into $t$ disjoint intervals $I_{\lambda_i}$ of length $\lambda_i$. If $\lambda_i=0$ then $I_{\lambda_i}$ consists of a single vertex.

The {\it weight of an interval} $I_{\lambda}$, denoted $\omega(I_{\lambda})$, is the sum of the weights of its vertices. For instance, the weights of the intervals in \Cref{d6t4.} are $\omega(I_{\lambda_4})=w_0$, $\omega(I_{\lambda_3})=w_1+w_2+w_3$, $\omega(I_{\lambda_2})=w_4$, and $\omega(I_{\lambda_1})=w_5+w_6$.

 \begin{figure}[H]
    \centerline{\xymatrix@R=0pt{
        \overset{w_0}{\bullet}&\overset{w_1}{\bullet}\ar@{-}[r]&\overset{w_2}{\bullet}\ar@{-}[r]&\overset{w_3}{\bullet}&\overset{w_4}{\bullet}&\overset{w_5}{\bullet}\ar@{-}[r]&\overset{w_6}{\bullet}\\
        I_{\lambda_4}&&I_{\lambda_3}&&I_{\lambda_2}&{}\save[]+<0.8cm,0cm>*\txt<8pc>{$\genedifflow\lambda 1$}\restore}
    }
    \caption{A decomposition for $d=6$, $t=4$, and $(\lambda_4,\lambda_3,\lambda_2,\lambda_1)=(0,2,0,1)$.}
    \label{d6t4.}
\end{figure}
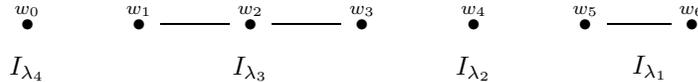 

\begin{definition}\label{def}
Define the arithmetic complex, $C_\bullet(\underline w)=C_\bullet(w_0,\dots, w_d)$, as follows: for each $k\in\Z$,
$$C_k(\underline w):=\underset{\sum_{i=1}^t \lambda_i=k}{\bigoplus} R\cdot f_{(\lambda_t,\ldots,\lambda_1)}\cong R^{\oplus\binom{d}{k}},$$
where $t=d-k+1\ge 1$, and $f_{(\lambda_t,\ldots,\lambda_1)}$ corresponds to a basis element of $C_k(\underline w)$. Note the basis elements of $C_k(\underline{w})$ are in bijective correspondence with decompositions of the path graph in \Cref{UWPG}. The differential $\partial_k$ is computed by removing a single edge from an interval of a basis element of $C_k(\underline w)$. More precisely, 
\begin{eqnarray*}
    \partial_k: C_k(\underline w)&\longrightarrow& C_{k-1}(\underline w)\\
    f_{(\lambda_t,\dots,\lambda_1)}&\mapsto &\sum_{j=1}^{t}\sum_{i=1}^{\lambda_j} (-1)^{t-j} \binom{\omega(I_{\lambda_j})}{\omega(I_{\lambda_j}\setminus\{i\})} f_{(\lambda_t,\dots,\lambda_{j+1},i-1,\lambda_j-i,\lambda_{j-1},\dots,\lambda_1)},
\end{eqnarray*}
where $I_\lambda\setminus\{i\}$ denotes the subinterval of $I_\lambda$ to the right of the $i$-th edge (counting starting at the left).  
\end{definition}

Our focus is when $w_0=x$ or $w_0=-x-2d$, and $w_1=\dots=w_d=1$ with the aim of establishing an isomorphism between $C_\bullet(x, 1^d)$ and $C_\bullet(-x-2d, 1^d)$.
\begin{remark} 
 When $w_0=-x-2d$ the following binomial identity is used in computing the differentials:
$$\binom{-x-m}{n}=(-1)^{n}\binom{x+m+n-1}{n}$$ for $m,n\in \Z$. We adopt the convention $\binom{x+m}{0}=1$ and $\binom{x+m}{n}=0$ for $n\in \Z_-$.

\end{remark}

\begin{example}\label{d3d4}
Computations with Macaulay2 \cite{M2} exemplify an isomorphism from $C_\bullet(x, 1^d)$ to $C_\bullet(-x-2d, 1^d)$ for $d=3$.

$$\begin{tikzcd}[ampersand replacement=\&]
0 \arrow{r} \& R \arrow{r}{\left(\begin{smallmatrix}
      \binom{x+3}{3}\\
      \binom{x+3}{2}\\
      x+3\\
      \end{smallmatrix}\right)} \arrow{d}[description]{(1)} \&[4em] R^3 \arrow{r}{\left(\begin{smallmatrix}
      -3&x+1&0\\
      -3&0&\binom{x+2}{2}\\
      0&-2&x+2\\
      \end{smallmatrix}\right)} \arrow{d}[description]{\left(\begin{smallmatrix}
          -1 & -2 & -1\\
          0 & 1 & 1\\
          0 & 0 & -1
      \end{smallmatrix}\right)} \&[7em] R^3 \arrow{r}{\left(\begin{smallmatrix}
      2,&-2,&x+1\\
      \end{smallmatrix}\right)} \arrow{d}[description]{\left(\begin{smallmatrix}
          -1 & 0 & -1\\
          0 & -1 & -3\\
          0 & 0 & 1
      \end{smallmatrix} \right)}\&[3em] R \arrow{d}[description]{(-1)} \arrow{r} \& 0\\[8ex]     
        0 \arrow {r} \& R \arrow{r}{\left(\begin{smallmatrix}
      -\binom{x+5}{3}\\
      \binom{x+4}{2}\\
      -x-3\\
      \end{smallmatrix}\right)} \& R^3 \arrow{r}{\left(\begin{smallmatrix}
      -3&-x-5&0\\
      -3&0&\binom{x+5}{2}\\
      0&-2&-x-4\\
      \end{smallmatrix}\right)} \& R^3 \arrow{r}{\left(\begin{smallmatrix}
      2,&-2,&-x-5\\
      \end{smallmatrix}\right)} \& R \arrow{r} \& 0
\end{tikzcd}$$
\end{example}

In \Cref{d3d4}, the isomorphism maps in homological degree $d-t+1$ can be obtained using the following formulas:

\noindent$t=1$: $f_{\lambda_1} \mapsto f_{\mu_1}$,

\noindent$t=2$: $f_{(\lambda_2,\lambda_1)} \mapsto {\displaystyle \sum_{\mu_1+\mu_2=d-1}(-1)^{d-\mu_2}} \binom{\mu_1}{\lambda_1} \cdot f_{(\mu_2, \mu_1)}$,

\noindent$t=3$: $f_{(\lambda_3,\lambda_2,\lambda_1)}\mapsto {\displaystyle  \sum_{\mu_1+\mu_2+\mu_3=d-2}(-1)^{d-\mu_3}} \binom{\mu_1}{\lambda_1} \binom{\lambda_1+\mu_1+\mu_2+2}{\lambda_1+\mu_1+\lambda_2+2} \cdot f_{(\mu_3,\mu_2,\mu_1)}$ .

A generalization of this pattern is given by ${\alpha_{\bullet}: C_\bullet(x,1^d) \longrightarrow C_\bullet(-x-2d,1^d)}$ where
for an arbitrary $d$ the map $\alpha_{d-t+1}: C_{d-t+1}(x,1^d) \longrightarrow C_{d-t+1}(-x-2d,1^d)$ is defined as 
\begin{eqnarray*}
&&\alpha_{d-t+1}(
f_{(\lambda_t,\lambda_{t-1},\dots,\lambda_1)})\\
&=& \sum_{\mu_1+\dots+\mu_t=d-t+1} (-1)^{d-\mu_{t}} \prod_{s=0}^{t-2} \binom{\sum_{j=1}^s (\lambda_j+\mu_j)+\mu_{s+1}+2s}{\sum_{j=1}^s (\lambda_j+\mu_j)+\lambda_{s+1}+2s}f_{(\mu_t,\mu_{t-1},\dots,\mu_1)}.
\end{eqnarray*}

The main result of the next section is that $\alpha_{\bullet}$ is indeed an isomorphism of complexes.

\section{Main results}\label{newresults}

This section is devoted to the proof of \Cref{main} which is broken into two parts. First, it is shown that $\alpha_{\bullet}$ defines a homomorphism of complexes (using lemmata \ref{lemmone} and \ref{bluecalculation}).  Next, \Cref{thirdlemma} shows that the matrices in $\alpha_{\bullet}$ are invertible by using an ordering of the basis, as explained in \Cref{basisrmk}. Finally, \Cref{uselesscorollary} provides a more specific description of those matrices.

\begin{lemma}\label{lemmone}
    Let $d,t \in \N$, $t\geq1$. For weak compositions $\lambda=(\lambda_t,\dots,\lambda_1)$ of $d-t+1$, and  $\mu=(\mu_{t+1},\mu_t,\dots,\mu_1)$ of $d-t$, define the following quantities:
    $$A_k(\lambda,\mu):=\prod_{s=0}^{k-2} \binom{\sum_{m=1}^{s} (\mu_m+\lambda_m)+\mu_{s+1}+2s} {\sum_{m=1}^{s} (\mu_m+\lambda_m)+\lambda_{s+1}+2s},$$
    $$B_k(\lambda,\mu):=\prod_{s=k+1}^{t-1} \binom{\sum_{m=1}^{s-1} (\mu_m+\lambda_m)+\mu_{s}+\mu_{s+1}+2s-1}{\sum_{m=1}^{s} (\mu_m+\lambda_m)+2s-1},$$
    $$g_k(\lambda,\mu;z):=\binom{\sum_{m=1}^{k-1} (\mu_m+\lambda_m)+\mu_{k}+2(k-1)}{\sum_{m=1}^{k-1} (\mu_m+\lambda_m)+\lambda_{k}-z+2(k-1)},$$
    $$h_k(\lambda,\mu;z) :=\binom{\sum_{m=1}^{k} (\mu_m+\lambda_m)+\mu_{k+1}-z+2(k-1)+2}{\sum_{m=1}^{k} (\mu_m+\lambda_m)+2(k-1)+1},$$
    where  $k=1,\dots,t$ and $z\in\Z$. Then the following hold
    \begin{enumerate}
        \item[$i$)] $g_1(\lambda,\mu;\lambda_1+1)=0;$
        \item[$ii$)] $A_{k}(\lambda,\mu)\cdot g_{k}(\lambda,\mu;0)=A_{k+1}(\lambda,\mu)$ and $h_{k}(\lambda,\mu;0)=g_{k+1}(\lambda,\mu;\lambda_{k+1} +1)$, for $1\le k\le t-1$;
        \item[$iii$)] $B_{k}(\lambda,\mu)=B_{k+1}(\lambda,\mu)\cdot h_{k+1}(\lambda,\mu;\lambda_{k+1} +1)$, for $1\le k\le t-2$;
        \item [$iv$)] $\A{t}\cdot \g{t}{0}=0$.
    \end{enumerate}
\end{lemma}

\begin{proof}
    The proof proceeds by explicit computations. Notice that $g_1(\lambda,\mu;\lambda_1+1)=\binom{\mu_1}{-1}=0$. Given $k=1,\dots,t-1$, part $ii)$ follows from the observations

    \thickmuskip=0mu
    \begin{align*}
        &A_{k}(\lambda,\mu)\cdot g_{k}(\lambda,\mu;0)=\prod_{s=0}^{k-2} \binom{\sum_{m=1}^{s} (\mu_m+\lambda_m)+\mu_{s+1}+2s} {\sum_{m=1}^{s} (\mu_m+\lambda_m)+\lambda_{s+1}+2s} \cdot \\
        &\binom{\sum_{m=1}^{k-1} (\mu_m{+}\lambda_m){+}\mu_{k}{+}2(k{-}1)}{\sum_{m=1}^{k-1} (\mu_m{+}\lambda_m){+}\lambda_{k}{+}2(k{-}1)}=\prod_{s=0}^{k-1} \binom{\sum_{m=1}^{s} (\mu_m{+}\lambda_m){+}\mu_{s+1}{+}2s} {\sum_{m=1}^{t} (\mu_m{+}\lambda_m){+}\lambda_{s+1}{+}2s}=A_{k+1}(\lambda,\mu)
    \end{align*}
    \thickmuskip=5mu plus 5mu
    and
    \begin{align*}
        g_{k+1}(\lambda,\mu;\lambda_{k+1} +1)&=\binom{\sum_{m=1}^{k} (\mu_m+\lambda_m)+\mu_{k+1}+2k}{\sum_{m=1}^{k} (\mu_m+\lambda_m)+2k-1}\\
        &=\binom{\sum_{m=1}^{k} (\mu_m+\lambda_m)+\mu_{k+1}+2(k-1)+2}{\sum_{m=1}^{k} (\mu_m+\lambda_m)+2(k-1)+1}=h_k(\lambda,\mu;0).
    \end{align*}
    
    \noindent Furthermore, if $k=1,\dots,t-2$ then
    \begin{align*}
        h_{k+1}(\lambda,\mu;\lambda_{k+1} +1)&= \binom{\sum_{m=1}^{k+1} (\mu_m+\lambda_m)+\mu_{k+2}-\lambda_{k+1}-1+2k+2}{\sum_{m=1}^{k+1} (\mu_m+\lambda_m)+2k+1}\\
        &=\binom{\sum_{m=1}^{k}(\mu_m+\lambda_m) +\mu_{k+1}+\mu_{k+2}+2k+1}{\sum_{m=1}^{k+1} (\mu_m+\lambda_m)+2k+1},
    \end{align*}
    
    \noindent and the equality $B_{k+1}(\lambda,\mu)\cdot h_{k+1}(\lambda,\mu;\lambda_{k+1} +1)=B_k(\lambda, \mu) $ is obtained.

    Finally, if $\A{t}=\prod_{s=0}^{t-2} \binom{\sum_{m=1}^{s} (\mu_m+\lambda_m)+\mu_{s+1}+2s} {\sum_{m=1}^{s} (\mu_m+\lambda_m)+\lambda_{s+1}+2s}\not=0$ then 

    $$\binom{\sum_{m=1}^{s} (\mu_m+\lambda_m)+\mu_{s+1}+2s} {\sum_{m=1}^{s} (\mu_m+\lambda_m)+\lambda_{s+1}+2s}=\binom{\sum_{m=1}^{s} (\mu_m+\lambda_m)+\mu_{s+1}+2s} {\mu_{s+1}-\lambda_{s+1}}\not=0$$

    \noindent for $s=0,\dots, t-2,$ or equivalently $\mu_j\geq\lambda_j$ for $j=1, \dots, t-1$.  Since $\sum_{j=1}^t\lambda_j=d-t+1$ and $\sum_{j=1}^{t+1}\mu_j=d-t$, it follows that $\mu_t-\lambda_t<0$ and thus
    \thickmuskip=0mu 
    $$g_t(\lambda,\mu;0)=\binom{\sum_{m=1}^{t-1} (\mu_m{+}\lambda_m){+}\mu_{t}{+}2(t{-}1)}{\sum_{m=1}^{t-1} (\mu_m{+}\lambda_m){+}\lambda_{t}{+}2(t{-}1)}=\binom{\sum_{m=1}^{t-1} (\mu_m{+}\lambda_m){+}\mu_{t}{+}2(t{-}1)}{\mu_t{-}\lambda_{t}}=0.$$
    \thickmuskip=5mu plus 5mu

    \noindent These observations yield the identity below which completes the proof of part $iv)$.
    \begin{equation*}
        \A{t}\cdot \g{t}{0}=
        \begin{cases}
            \A{t}\cdot 0 & \text{ if } \A{t}\not=0\\
            0\cdot \g{t}{0} & \text{ if } \A{t}=0
        \end{cases}
        \,\,=0. \qedhere
    \end{equation*}
\end{proof}

\begin{lemma}\label{bluecalculation}
    Under the assumptions of \Cref{lemmone}
    \begin{align*}
      &\sum_{k=1}^{t-1}(-1)^{t-k}\A{k}\cdot\B{k}\cdot \Big[\h{k}{\lambda_{k}+1}\cdot \g{k}{\lambda_{k}+1}\\
      &+\g{k}{0}\cdot \h{k}{0}\Big]=-\A{t}\cdot \g{t}{\lambda_{t}+1}.
    \end{align*}
    \end{lemma}
    \begin{proof}
    For $k=1,\dots, t-2$, \Cref{lemmone} $ii)$ and $iii)$ imply
    \begin{eqnarray*}
    &&\A{k}\cdot\B{k}\cdot \g{k}{0}\cdot \h{k}{0}\\
    &=&\A{k+1}\cdot\B{k+1}\cdot \g{k+1}{\lambda_{k+1}+1}\cdot \h{k+1}{\lambda_{k+1}+1}.
    \end{eqnarray*}
    In the case $k=t-1$, $\B{t-1}=1$ and so by \Cref{lemmone} $ii)$ the following equation holds:
    \begin{equation*}
        \A{t-1}\cdot\B{t-1}\cdot \g{t-1}{0}\cdot \h{t-1}{0}=\A{t}\cdot \g{t}{\lambda_{t}+1}.
    \end{equation*}
    The following chain of equalities finish the proof, where $(*)$ uses \Cref{lemmone} $i)$:
    \begin{align*}
        &\sum_{k=1}^{t-1}(-1)^{t-k}\A{k}\cdot\B{k}\cdot \g{k}{0}\cdot \h{k}{0}\\
        =&\sum_{k=1}^{t-2}(-1)^{t-k}\A{k}\cdot\B{k}\cdot \g{k}{0}\cdot \h{k}{0}\\
        &-\A{t-1}\cdot\B{t-1}\cdot \g{t-1}{0}\cdot \h{t-1}{0}\\
        =&\sum_{k=1}^{t-2}(-1)^{t-k}\A{k+1}\cdot\B{k+1}\cdot \h{k+1}{\lambda_{k+1}+1}\cdot \g{k+1}{\lambda_{k+1}+1}\\
        &-\A{t}\cdot \g{t}{\lambda_{t}+1}\\
        \overset{*}{=}&-\sum_{k=1}^{t-1}(-1)^{t-k}\A{k}\cdot\B{k}\cdot \h{k}{\lambda_{k}+1}\cdot \g{k}{\lambda_{k}+1}\\
        &-\A{t}\cdot \g{t}{\lambda_{t}+1}.\qedhere
    \end{align*}
\end{proof}

We are ready to prove \Cref{main}:

\begin{proof}[Proof of \Cref{main}] 
Consider the map
$\alpha^d_{d-t+1}: C_{d-t+1}(x,1^d) \longrightarrow  C_{d-t+1}(-x-2d,1^d)$ defined as
\begin{eqnarray*}
&&\alpha^d_{d-t+1}(
f_{(\lambda_t,\lambda_{t-1},\dots,\lambda_1)})\\
&=& \sum_{\mu_1+\dots+\mu_t=d-t+1} (-1)^{d-\mu_{t}} \prod_{s=0}^{t-2} \binom{\sum_{j=1}^s (\lambda_j+\mu_j)+\mu_{s+1}+2s}{\sum_{j=1}^s (\lambda_j+\mu_j)+\lambda_{s+1}+2s}f_{(\mu_t,\mu_{t-1},\dots,\mu_1)}
\end{eqnarray*}
which is a homomorphism for $t\geq 1$, presented by a $\binom{d}{t-1}\times \binom{d}{t-1}$ matrix. When there is no ambiguity the superscript $d$ is removed and the map is denoted by $\alpha_{d-t+1}$. To prove $\alpha_\bullet$ is an isomorphism of complexes entails demonstrating $\alpha_\bullet$ is a morphism of chain complexes and each $\alpha_{d-t+1}$ is represented by an invertible matrix. The former is established first, that is
\begin{equation}\label{equofmain}
    \alpha_{d-t}(\partial_{d-t+1}^x(f_{(\lambda_t,\dots, \lambda_1)}))=\partial_{d-t+1}^{-x-2d}(\alpha_{d-t+1}(f_{(\lambda_t, \dots, \lambda_1)}))
\end{equation}
for each $t$ and basis element $f_{(\lambda_t,\dots, \lambda_1)}$ of $C_{d-t+1}(x,1^d)$. \Cref{thirdlemma} addresses the invertibility of $\alpha_{d-t+1}$. Taking $\{f_{(\mu_{t+1}, \dots,  \mu_{1})}\}$ as a basis of $C_{d-t}(-x-2d,1^d)$ results in equalities
\begin{align*}
    \partial_{d-t+1}^{-x-2d}(\alpha_{d-t+1}(f_{(\lambda_t, \dots, \lambda_1)})) &= \sum_{\mu_{1} + \dots+ \mu_{t+1} = d - t}c_{(\mu_{t+1}, \dots , \mu_{1})}f_{(\mu_{t+1}, \dots,  \mu_{1})},\\
    \alpha_{d-t}(\partial_{d-t+1}^x(f_{(\lambda_t,\dots, \lambda_1)})) &= \sum_{\mu_{1} + \dots+ \mu_{t+1} = d - t}\tilde c_{(\mu_{t+1}, \dots , \mu_{1})}f_{(\mu_{t+1}, \dots,  \mu_{1})},
\end{align*}
for some coefficients $c_{(\mu_{t+1}, \dots , \mu_{1})}$ and $\tilde c_{(\mu_{t+1}, \dots , \mu_{1})}$. \Cref{equofmain} is proved by fixing $\mu_{t+1}, \dots, \mu_{1}$ and the following sequence of steps, the details of which constitute the remainder of this proof.
\begin{enumerate}
    \item[Step 1:]Compute $\partial_{d-t+1}^{-x-2d}(\alpha_{d-t+1}(f_{(\lambda_t, \dots, \lambda_1)}))$ and $c_{(\mu_{t+1}, \dots , \mu_{1})}$.
    \item[Step 2:] Compute $\alpha_{d-t}(\partial_{d-t+1}^x(f_{(\lambda_t,\dots, \lambda_1)}))$ and $\tilde c_{(\mu_{t+1}, \dots , \mu_{1})}$.
    \item[Step 3:] Compare the coefficients $c_{(\mu_{t+1}, \dots , \mu_{1})}$ and $\tilde c_{(\mu_{t+1}, \dots , \mu_{1})}$.
\end{enumerate}
\noindent\textbf{Step 1.} Let us first determine the basis elements $f_{(\epsilon_t, \dots , \epsilon_1)}$ of $C_{d-t+1}(-x-2d,1^d)$, such that in the sum for $\partial_{d-t + 1}^{-x-2d}(f_{(\epsilon_t, \dots \epsilon_{1})})$, the basis element $f_{(\mu_{t+1}, \dots , \mu_{1})}$ appears with non-zero coefficient. By definition of $\partial_{d-t+1}^{-x-2d}$ such an $f_{(\epsilon_t, ..., \epsilon_{1})}$ is of the form $$f_{(\mu_{t+1}, ...., \mu_{k+2}, \mu_{k+1} + \mu_{k}+1, \mu_{k-1}, \dots,\mu_{1})}$$ for some $1 \leq k \leq t$.
More concretely, for $k = t$, the coefficient of $f_{(\mu_{t+1}, \dots, \mu_{1})}$ in the sum $\partial_{d-t+1}^{-x-2d}(f_{(\mu_{t+1} + \mu_t + 1, \mu_{t-1}, \dots, \mu_{1})})$ is
\[\binom{-x-2d+\mu_{t+1} + \mu_t + 1}{\mu_t + 1} = (-1)^{\mu_t + 1} \binom{x + 2d - \mu_{t+1}-1}{\mu_t + 1}.\]
For $k < t$, in the sum $\partial_{d-t+1}^{-x-2d}(f_{(\mu_{t+1}, ...., \mu_{k+2}, \mu_{k+1} + \mu_{k}+1, \mu_{k-1}, \dots,\mu_{1})})$ the coefficient of $f_{(\mu_{t+1}, \dots, \mu_{1})}$ is given by
$$(-1)^{t-k}\binom{\mu_{k+1} + \mu_{k} + 2}{\mu_{k}+1}.$$

The coefficient of $f_{(\mu_{t+1}, ...., \mu_{k+2}, \mu_{k+1} + \mu_{k}+1, \mu_{k-1}, \dots,\mu_{1})}$ in the sum for $\alpha_{d-t+1}(f_{(\lambda_{t}, \dots, \lambda_{1})})$ is computed next. 
For $k < t$, the coefficient is given by 
    \[(-1)^{d-\mu_{t+1}}\cdot\prod_{s=k+1}^{t-1} \binom{\sum_{m=1}^{s-1} (\mu_m+\lambda_m)+\mu_{s}+\mu_{s+1}+2s-1}{\sum_{m=1}^{s} (\mu_m+\lambda_m)+2s-1}\]
    \[\prod_{s=0}^{k-2} \binom{\sum_{m=1}^{s} (\mu_m+\lambda_m)+\mu_{s+1}+2s} {\sum_{m=1}^{s} (\mu_m+\lambda_m)+\lambda_{s+1}+2s}\cdot \binom{\sum_{m=1}^{k-1} (\mu_m+\lambda_m)+\mu_{k}+\mu_{k+1}+2k - 1}{\sum_{m=1}^{k-1} (\mu_m+\lambda_m) + \lambda_{k}+2k-2}\]
    \[=(-1)^{d-\mu_{t+1}}\B k \cdot \A{k} \cdot \binom{\sum_{m=1}^{k-1} (\mu_m+\lambda_m)+\mu_{k}+\mu_{k+1}+2k - 1}{\sum_{m=1}^{k-1} (\mu_m+\lambda_m) + \lambda_{k}+2k-2}\]
where $A_k, B_k$ are as defined in Lemma \ref{lemmone}.

The coefficient for $k = t$ is
    \[(-1)^{d-(\mu_{t+1}+\mu_t+1)}\prod_{s=0}^{t-2} \binom{\sum_{m=1}^{s} (\mu_m+\lambda_m)+\mu_{s+1}+2s} {\sum_{m=1}^{s} (\mu_m+\lambda_m)+\lambda_{s+1}+2s}= (-1)^{d - (\mu_{t+1} + \mu_t + 1)}\A t.\]

\noindent Thus \Cref{da} displays the coefficient of $f_{(\mu_{t+1}, \mu_t, \dots, \mu_1)}$ in the composition $\partial_{d-t+1}^{-x-2d}(\alpha_{d-t+1}(f_{\lambda_t, \dots, \lambda_1}))$.

\begin{equation}\label{da}
\begin{gathered}
c_{(\mu_{t+1}, \dots, \mu_1)}=(-1)^{d-\mu_{t+1}}\Bigg[\sum_{k = 1}^{t-1}\Bigg((-1)^{t-k}\binom{\sum_{m=1}^{k-1} (\mu_m+\lambda_m)+\mu_{k}+\mu_{k+1}+2k - 1}{\sum_{m=1}^{k-1} (\mu_m+\lambda_m) + \lambda_{k}+2k-2}\cdot \\
\cdot \binom{\mu_{k+1} + \mu_{k} + 2}{\mu_{k}+1}\cdot \A k \cdot \B k\Bigg) + \binom{x + 2d - \mu_{t+1}-1}{\mu_t + 1}\A t\Bigg].
\end{gathered} 
\end{equation}

\noindent\textbf{Step 2.} By \Cref{def}
\begin{align*}
    \partial^x_{d-t+1}(f_{(\lambda_t,\dots,\lambda_1)})&=\sum_{i=1}^{\lambda_t} \binom{x+\lambda_t}{\lambda_t -i+1} f_{(i-1,\lambda_t-i,\lambda_{t-1},\dots,\lambda_1)}\\
    &+\sum_{k=1}^{t-1}\sum_{i=1}^{\lambda_k} (-1)^{t-k} \binom{\lambda_k +1}{i} f_{(\lambda_t,\dots,\lambda_{k+1},i-1,\lambda_k-i,\lambda_{k-1},\dots,\lambda_1)}.
\end{align*}
\noindent Applying $\alpha_{d-t}$ to the basis elements appearing in the above summation, it follows 
\begin{align*}
    &\alpha_{d-t}(f_{(i-1,\lambda_t-i,\lambda_{t-1},\dots,\lambda_1)})\\
    =&\sum_{\mu_1+\dots+\mu_{t+1}=d-t}(-1)^{d-\mu_{t+1}} \prod_{s=0}^{t-2} \binom{\left(\sum_{t=1}^s \lambda_t+\mu_t\right)+\mu_{s+1}+2s}{\left(\sum_{t=1}^s \lambda_t+\mu_t\right)+\lambda_{s+1}+2s}\cdot\\
    &\cdot\binom{\sum_{m=1}^{t-1} (\mu_m+\lambda_m)+\mu_{t}+2(t-1)}{\sum_{m=1}^{t-1} (\mu_m+\lambda_m)+\lambda_{t}-i+2(t-1)}f_{(\mu_{t+1},\mu_t,\dots,\mu_1)}\\
    =&\sum_{\mu_1+\dots+\mu_{t+1}=d-t}(-1)^{d-\mu_{t+1}}\A t\cdot \g{t}{i}f_{(\mu_{t+1},\mu_t,\dots,\mu_1)},
    \end{align*}
    and
    \begin{align*}
    &\alpha_{d-t}(f_{(\lambda_t,\dots,\lambda_{k+1},i-1,\lambda_k-i,\lambda_{k-1},\dots,\lambda_1)})\\
    =&\sum_{\mu_1+\dots+\mu_{t+1}=d-t} (-1)^{d-\mu_{t+1}}\prod_{s=0}^{k-2} \binom{\left(\sum_{t=1}^s \lambda_t+\mu_t\right)+\mu_{s+1}+2s}{\left(\sum_{t=1}^s \lambda_t+\mu_t\right)+\lambda_{s+1}+2s}\cdot\\
    &\cdot\binom{\sum_{m=1}^{k-1} (\mu_m+\lambda_m)+\mu_{k}+2(k-1)}{\sum_{m=1}^{k-1} (\mu_m+\lambda_m)+\lambda_{k}-i+2(k-1)}\cdot\binom{\sum_{m=1}^{k} (\mu_m+\lambda_m)+\mu_{k+1}-i+2k}{\sum_{m=1}^{k} (\mu_m+\lambda_m)+2k-1}\cdot\\
    &\cdot\prod_{s=k+1}^{t-1} \binom{\sum_{m=1}^{s-1} (\mu_m+\lambda_m)+\mu_{s}+\mu_{s+1}+2s-1}{\sum_{m=1}^{s} (\mu_m+\lambda_m)+2s-1}\,f_{(\mu_{t+1},\mu_t,\dots,\mu_1)}\\
    =&\sum_{\mu_1+\dots+\mu_{t+1}=d-t}(-1)^{d-\mu_{t+1}}\A{k}\cdot \g{k}{i}\cdot \h{k}{i}\cdot \B{k} f_{(\mu_{t+1},\mu_t,\dots,\mu_1)},
\end{align*}
where $\g{k}{-}$, and $\h{k}{-}$ are defined in Lemma \ref{lemmone}. The above computations result in

\begin{align*}
   & (\alpha_{d-t}\circ\partial^x_{d-t+1})(f_{(\lambda_t,\dots,\lambda_1)})
    =\sum_{i=1}^{\lambda_t} \binom{x+\lambda_t}{\lambda_t -i+1} \alpha_{d-t}(f_{(i-1,\lambda_t-i,\lambda_{t-1},\dots,\lambda_1)})\notag\\
    &+\sum_{k=1}^{t-1}\sum_{i=1}^{\lambda_k} (-1)^{t-k} \binom{\lambda_k +1}{i} \alpha_{d-t}(f_{(\lambda_t,\dots,\lambda_{k+1},i-1,\lambda_k-i,\lambda_{k-1},\dots,\lambda_1)})\\
&=\sum_{\mu_1+\dots+\mu_{t+1}=d-t} \tilde c_{(\mu_{t+1}, \dots, \mu_1)}f_{(\mu_{t+1},\mu_t,\dots,\mu_1)},\nonumber
\end{align*}
where
\begin{equation}\label{ad}
\begin{gathered}
     \tilde c_{(\mu_{t+1}, \dots, \mu_1)}=(-1)^{d-\mu_{t+1}}\Bigg[\sum_{i=1}^{\lambda_t} \binom{x+\lambda_t}{\lambda_t -i+1}\A{t}\cdot \g{t}{i}\\
    +\sum_{k=1}^{t-1}(-1)^{t-k}\Bigg(\sum_{i=1}^{\lambda_k}  \binom{\lambda_k +1}{i} \A{k}\cdot \g{k}{i}\cdot \B{k}\cdot \h{k}{i}\Bigg)\Bigg].
    \end{gathered}
\end{equation}

\noindent\textbf{Step 3.} The final step is to exhibit the coefficient $\tilde c_{(\mu_{t+1}, \dots, \mu_1)}$ of $f_{(\mu_{t+1}, \dots,\mu_1)}$ in \Cref{ad} is equal to $c_{(\mu_{t+1}, \dots, \mu_1)}$, i.e. the expression in \Cref{da}. This is proven with the aid of the following combinatorial identities.
\begin{equation}\label{Vi}
    \binom{m+n}{r} = \sum_{k \geq 0}\binom{m}{k}\cdot \binom{n}{r-k} \quad\text{ with $m,n, r\in\N$}.
\end{equation}
\begin{equation} \label{firstlemmone}
        \sum_{j\ge 0} \binom{b}{j} \binom{y+a}{a-b+1+j} \binom{y+a+b+c-j}{c-j} = \binom{y+a+c}{a-b+c+1} \binom{a+c+1}{c}
\end{equation}
where $y$ is a variable and $a,b,c\in \Z_+$.
Indeed, \Cref{Vi} is the well-known Vandermonde's identity and \Cref{firstlemmone} is the content of \cite[Corollary 4]{szekely}. These formulas along with \Cref{lemmone} are used to rewrite the inner sums in $\tilde c_{(\mu_{t+1}, \dots, \mu_1)}$ as having index $i\geq 0$, which are later recognized as a product of binomial coefficients. By the above discussion and the fact that $\binom{x + \lambda_t}{\lambda_t-i + 1} = 0$ for $i > \lambda_t + 1$, 
\thickmuskip=0mu
\begin{align*}
&\sum_{i=1}^{\lambda_t} \binom{x+\lambda_t}{\lambda_t -i+1}\A{t}\cdot \g{t}{i}\\
=&\A{t}\Bigg(\sum_{i\geq0}
\binom{x{+}\lambda_t}{\lambda_t{-}i{+}1} \g{t}{i}-\g{t}{\lambda_t {+}1} \Bigg) {-}\binom{x{+}\lambda_t}{1{+}\lambda_t} \underbrace{\A{t}\g{t}{0}}_{=0 \text{ by \Cref{lemmone} $iv)$}}\\
=&\A{t}\Bigg(\underbrace{\sum_{i\geq 0} \binom{x+\lambda_t}{\lambda_t -i+1}\binom{\sum_{m=1}^{t-1} (\mu_m+\lambda_m)+\mu_{t}+2(t-1)}{\mu_t-\lambda_t +i}}_{=\binom{\sum_{m=1}^{t-1} (\mu_m+\lambda_m)+\mu_{t}+\lambda_t+x+2(t-1)}{\mu_t+1} \text{ by \Cref{Vi}}}-\g{t}{\lambda_t +1}\Bigg)\\
=&\A{t} \Bigg(\binom{x+2d-\mu_{t+1}-1}{\mu_t+1}-\g{t}{\lambda_t +1}\Bigg),
\end{align*}
\thickmuskip=5mu plus 5mu
as $\lambda_1+\dots+\lambda_t=d-t+1=\mu_1+\dots+\mu_{t+1}+1.$  
For the other summation appearing,
\begin{align*}
&\sum_{k=1}^{t-1}(-1)^{t-k}\left[\sum_{i=1}^{\lambda_k}  \binom{\lambda_k +1}{i} \A{k} \g{k}{i} \B{k} \h{k}{i}\right]\\
=&\sum_{k=1}^{t-1}(-1)^{t-k}\A{k}\B{k}\left[\sum_{i\geq 0}  \binom{\lambda_k +1}{i}  \g{k}{i} \h{k}{i} \right. \\
-& \left.  \g{k}{0} \h{k}{0}-\g{k}{\lambda_k +1} \h{k}{\lambda_k +1} \right]\\
=&\sum_{k=1}^{t-1}(-1)^{t-k}\A{k}\B{k}\left[\sum_{i\geq 0}  \binom{\lambda_k +1}{i}  \g{k}{i} \h{k}{i}\right ]\\
{-}& \sum_{k=1}^{t-1}(-1)^{t-k}\A{k}\B{k}\left[ \g{k}{0} \h{k}{0}{+}\g{k}{\lambda_k{+}1} \h{k}{\lambda_k{+}1} \right]\\
=& \sum_{k=1}^{t-1}(-1)^{t-k}\A{k}\B{k}\left[\sum_{i\geq 0}  \binom{\lambda_k+1}{i}  \g{k}{i} \h{k}{i}\right]\\
+&\A{t} \g{t}{\lambda_{t}+1}
\end{align*}
\thickmuskip=5mu plus 5mu
where the last equality follows by \Cref{bluecalculation}. Moreover, by applying \Cref{firstlemmone} with   \thickmuskip=0mu $y=\sum_{m=1}^{k-1} (\mu_m+\lambda_m)+2(k-1)$, $a=\mu_{k}$, $b=\lambda_{k}+1$ and $c=\mu_{k+1}+1$,  we obtain
$$\sum_{i\geq 0}  \binom{\lambda_k {+}1}{i}  \g{k}{i}\cdot \h{k}{i}=\sum_{i\geq 0} \binom{\lambda_k {+}1}{i} \binom{\sum_{m=1}^{k-1} (\mu_m{+}\lambda_m)+2(k-1) {+}\mu_{k}}{\mu_k{-}\lambda_{k}{+}i}$$

\begin{align*}
    &\cdot\binom{\sum_{m=1}^{k-1} (\mu_m{+}\lambda_m)+2(k-1)+\mu_k+(\lambda_k+1)+(\mu_{k+1}+1)-i}{\mu_{k+1}+1-i}\\
    &=\binom{\sum_{m=1}^{k-1} (\mu_m+\lambda_m)+\mu_{k}+\mu_{k+1}+2k-1}{\mu_k-\lambda_{k}+\mu_{k+1}+1} \cdot\binom{\mu_k+\mu_{k+1}+2}{\mu_{k+1}+1}.
\end{align*}

\noindent It follows that
\begin{eqnarray*}
&&\sum_{k=1}^{t-1}(-1)^{t-k}\Big(\sum_{i=1}^{\lambda_k}  \binom{\lambda_k +1}{i} \A{k}\cdot \g{k}{i}\cdot \B{k}\cdot \h{k}{i}\Big)\\
&=& \sum_{k=1}^{t-1}(-1)^{t-k}\A{k}\cdot\B{k}\Bigg(\binom{\sum_{m=1}^{k-1} (\mu_m+\lambda_m)+\mu_{k}+\mu_{k+1}+2k-1}{\mu_k-\lambda_{k}+\mu_{k+1}+1}\\
&&\cdot\binom{\mu_k+\mu_{k+1}+2}{\mu_{k+1}+1}\Bigg)+\A{t}\cdot \g{t}{\lambda_{t}+1}
\end{eqnarray*}
\noindent and the claimed equality below concludes that $\alpha_\bullet$ is a morphism of chain complexes.
\[\tilde c_{(\mu_{t+1}, \dots, \mu_1)}=(-1)^{d-\mu_{t+1}}\Big[\A{t}\binom{x{+}2d{-}\mu_{t+1}{-}1}{\mu_t+1}{+}\sum_{k=1}^{t-1}(-1)^{t-k}\A{k}\B{k}\]
\[\cdot\binom{\sum_{m=1}^{k-1} (\mu_m+\lambda_m)+\mu_{k}+\mu_{k+1}+2k-1}{\mu_k-\lambda_{k}+\mu_{k+1}+1}\binom{\mu_k+\mu_{k+1}+2}{\mu_{k+1}+1}\Big]=c_{(\mu_{t+1}, \dots, \mu_1)}. \qedhere\]
\end{proof}

The following remark introduces an order on the basis of $\displaystyle C_k(\underline w)$ to facilitate understanding of the matrices corresponding to the maps constructed in \Cref{main}.

\begin{remark}\label{basisrmk}
Let $t, k \in \N$, $t>0$.  Define an order $\order$ on the set of weak compositions of $k$ with $t$ parts
$$S_{t,k}=\left\{(a_t, a_{t-1}\dots, a_1)\in\N^t\mid \sum_{i=1}^ta_i=k\right\}$$ 
as follows: for  $a=(a_t,\ldots,a_1)$, $b=(b_t,\ldots,b_1)\in S_{t,k}$, 
consider the associated monomials $x_1^{a_1}\cdots x_t^{a_t}$ and $x_1^{b_1}\cdots x_t^{b_t}$. Define
$a\order b$ if $x_1^{a_1}\cdots x_t^{a_t} \geq_{\text{RLex}} x_1^{b_1}\cdots x_t^{b_t}$ with respect to the reverse lexicographic order on the monomials. In other words,
$a\order b$ if $a=b$ or the first non-zero entry of the vector of integers $a-b=(a_t-b_t,\ldots,a_1-b_1)$ is negative.

For each $d \in \N$ and for each $\underline w=(w_0,\dots,w_d)\in R\times \N^{d-1}$, denote the basis of $C_{d-t+1}(\underline w)$ by
$$\BC(C_{d-t+1}(w_0, w_1, \dots, w_d))=\{ f_{a^{\ugly{1}}},f_{a^\ugly 2},\dots\mid a^\ugly i\in S_{t, d-t+1}\},$$ where superscript $[i]$ denotes the $i$-th basis element under the ordering $a^\ugly i\order a^\ugly{i+1}$ for all $i\geq 1$. For example,Robinson-Schensted correspondence. The article contains a section treating the basic results about the passage to initial ideals and algebras. $a^\ugly1=(0,\dots, 0, d-t+1)$ is the biggest element in $S_{t,d-t+1}$ and $a^\ugly{\binom{d}{t-1}}=(d-t+1, 0,\dots, 0)$ is the smallest one. 
Fix $w_0$ and let $d=4$ and $t=3$. The basis of $C_2(w_0,\dots, w_4)$ is ordered as 
$$\BC(C_2(w_0,\dots, w_4))=\{f_{(0,0,2)}, f_{(0,1,1)}, f_{(0,2,0)}, f_{(1,0,1)}, f_{(1,1,0)}, f_{(2,0,0)}\}.$$
In particular, denote the bases of $C_{d-t+1}(x,1^d)$ and $C_{d-t+1}(-x-2d,1^d)$ by 
\begin{eqnarray*}
    \BC(C_{d-t+1}(x,1^d))&=&\{ f_{\lambda^{\ugly{1}}},f_{\lambda^\ugly 2},\dots \mid \lambda^\ugly i\in S_{t, d-t+1}\}, \\ 
    \BC(C_{d-t+1}(-x-2d,1^d))&=&\{ f_{\mu^\ugly 1},f_{\mu^\ugly 2},\dots \mid \mu^\ugly j\in S_{t, d-t+1}\}.    
\end{eqnarray*}

Since $\lambda^\ugly i$ and $\mu^\ugly j$ are in $S_{t, d-t+1}$ for each $i,j\geq 1$, they are comparable and
$\lambda^\ugly i=\mu^\ugly j$ if and only if $i= j$. Whence $\lambda^\ugly i\order\mu^\ugly j$ if and only if $i\leq j$. Moreover for the map $\alpha^{}_{d-t+1}$ constructed in \Cref{main},
the $(i,j)$-th entry of the matrix of $\alpha^{}_{d-t+1}$ with respect to the above order is ${\aij ij=[\alpha_{d-t+1}]_{(f_{\lambda^{\ugly i}}, f_{\mu^{\ugly j}})}}$ for each $1 \leq i,j\leq \binom{d}{t-1}$, that is,
\begin{equation*}
    \alpha^{}_{d-t+1}=
    \begin{bmatrix}
        \aij11 & \aij21&\dots & \aij i1& \dots\\
        \aij12& \aij22&\dots &\aij i2& \dots\\
        \vdots &\vdots&&\vdots&\\
        \aij 1j&\aij 2j&\dots&\aij ij&\dots\\
        \vdots &\vdots&&\vdots&\\
    \end{bmatrix}.
\end{equation*}
\end{remark}

\begin{lemma}\label{thirdlemma}
With respect to the basis order introduced in \Cref{basisrmk}, the matrices corresponding to the maps $\alpha_{d-t+1}$, $1\le t\le d+1$, constructed in \Cref{main} are upper triangular, with diagonals made up of $1$'s and $-1$'s, and hence are invertible.
\end{lemma}
\begin{proof}
Let $\lambda=(\lambda_t,\lambda_{t-1},\dots,\lambda_1)$ and $\mu=(\mu_t,\mu_{t-1},\dots,\mu_1)$ be compositions such that $f_\lambda\in \BC(C_{d-t+1}(x,1^d))$  and $f_\mu\in\BC(C_{d-t+1}(-x-2d,1^d))$. 
Observe that
$$[\alpha_{d-t+1}]_{(f_{\lambda}, f_{\mu})}=
(-1)^{d-\mu_{t}} \prod_{s=0}^{t-2} \binom{\sum_{j=1}^s (\lambda_j+\mu_j)+\mu_{s+1}+2s}{\sum_{j=1}^s (\lambda_j+\mu_j)+\lambda_{s+1}+2s}\not=0$$
if and only if $\binom{\sum_{j=1}^s (\lambda_j+\mu_j)+\mu_{s+1}+2s}{\mu_{s+1}-\lambda_{s+1}}\not=0$ for $s=0,\dots, t-2$. This is equivalent to 
$\mu_j \geq \lambda_j $ for $j=1,\dots,t-1$, which means 
\begin{equation}\label{equation1}
    \lambda=\mu\text{\quad or \quad} 
    \begin{cases}
     \mu_j \geq \lambda_j  & \text{for each } 1\leq j\leq t-1,\\
    \mu_j > \lambda_j
    &\text{for some } 1\leq j\leq t-1,\text{ and} \\
    \mu_t< \lambda_t
    \end{cases}
\end{equation}
as $\sum_{i=1}^t\lambda_i=\sum_{i=1}^t\mu_i$.
Hence by the definition of $\order$, $[\alpha_{d-t+1}]_{(f_{\lambda}, f_{\mu})}\not=0$ implies $\nonzero$. Consequently
\begin{equation*}
    [\alpha_{d-t+1}]_{(f_{\lambda}, f_{\mu})}=
    \begin{cases}
    0 & \text{if } \lambda\order \mu\text{ and } \mu\not = \lambda,\\
    (-1)^{d-\mu_t} & \text{if } \mu= \lambda.
    \end{cases}
\end{equation*}
Thus for $f_{\lambda^{\ugly i}}$ 
and $f_{\mu^{\ugly j}}$, 
the identity 
\begin{equation*}
    [\alpha_{d-t+1}]_{(f_{\lambda^{\ugly i}}f_{\mu^{\ugly j}})}=
    \begin{cases}
    0 & \text{if } i<j,\\
    1 \text{ or }-1  & \text{if } i=j
    \end{cases}
\end{equation*}
proves $\alpha_{d-t+1}$ is upper triangular, with diagonals made up of $1$'s and $-1$'s.
\end{proof}

\begin{corollary}\label{uselesscorollary}
Using the basis order introduced in \Cref{basisrmk}, and for each integer $ 1 \leq t \leq d+1 $, the matrices corresponding to the maps constructed in \Cref{main} are of the form
\begin{equation*}
        \alpha^d_{d-t+1}=\left(
        \begin{array}{c|c}
          (-1)^d\mathbbm{1}_{\binom{d-1}{t-2}}    &  \gamma_{d-t+1}^d \Tstrut\Bstrut\\
        \hline
           \bigzero  & \alpha^{d-1}_{d-t} \Tstrut\Bstrut
        \end{array}
        \right)
    \end{equation*}
with $\gamma_{d-t+1}^d$ a $\binom{d-1}{t-2} \times \binom{d-1}{t-1}$ matrix, and $\mathbbm{1}_{\binom{d-1}{t-2}}$ is the  identity matrix of size ${\binom{d-1}{t-2}}$.
\end{corollary}

\begin{remark}{\label{mapconeremark}}
The blocks on the left and right correspond to the basis elements $f_{\lambda}$ in the domain with $\lambda_t = 0$ and $\lambda_t \neq 0$ respectively. Likewise, the blocks on top and bottom correspond to the basis elements $f_{\mu}$ in the codomain satisfying $\mu_t = 0$ and $\mu_t \neq 0$ respectively. 
This decomposition reflects the fact that $C_{\bullet}(x, 1^d)$ is the mapping cone of a morphism $\phi_x: C_{\bullet}(x+1, 1^{d-1}) \longrightarrow C_{\bullet}(1^d)$ as described in \cite{raicukeller}. There the proof is for $w_0$ an integer, but it extends immediately to the case $w_0$ a polynomial (checking this requires showing the commutativity of some diagrams, which if true for infinitely many specializations to $w_0$ an integer must also be true for $w_0$ a polynomial). As a basic property of mapping cones, the fact that $\alpha_{\bullet}^d$ is a morphism of complexes means that the following diagram commutes up to homotopy given by $\gamma_{\bullet}$:
$$\begin{tikzcd}
    C_{\bullet}(x + 1, 1^{d-1}) \arrow[r, "\alpha^{d-1}_{\bullet}"] \arrow[d, "\phi_x"] & C_{\bullet}(-x -2d+1, 1^{d-1}) \arrow[d, "\phi_{-x-2d}"]\\
    C_{\bullet}(1^d) \arrow[r, "(-1)^d"] & C_{\bullet}(1^d)
\end{tikzcd}$$
Note that $\alpha_{\bullet}^{d-1}: C_{\bullet}(x+1, 1^{d-1}) \to C_{\bullet}(-x-2d+1, 1^{d-1})$ makes sense as the formula for $\alpha_{\bullet}^{d-1}$ is independent of $x$.

Eventually, using induction on $d$, \Cref{uselesscorollary} implies that $\alpha^d_k$ has determinant $\pm 1$ giving a second proof of its invertibility.
\end{remark}

\begin{proof}[Proof of \Cref{uselesscorollary}]
Let $\alpha^{d-1}_{d-t}: C_{d-t}(x,1^{d-1})\longrightarrow C_{d-t}(-x-2d+2,1^{d-1})$ be as defined in Theorem \ref{main}. For $f_{\lambda}, f_{\mu}$ satisfying $\lambda_t, \mu_t \geq 1$, by the definition of $\alpha_{\bullet}^d$

\begin{align*}
&[\alpha^d_{d-t+1}]_{(f_{\lambda}, f_{\mu})}=\underbrace{(-1)^{d-\mu_{t}}}_{(-1)^{(d-1)-(\mu_{t}-1)} } \prod_{s=0}^{t-2} \binom{\sum_{j=1}^s (\lambda_j+\mu_j)+\mu_{s+1}+2s}{\sum_{j=1}^s (\lambda_j+\mu_j)+\lambda_{s+1}+2s} \\
&= [\alpha^{d-1}_{d-t}]_{(f_{(\lambda_t-1, \lambda_{t-1}, ... ,\lambda_1)}, f_{(\mu_t - 1, \mu_{t-1}, ... , \mu_1)})}
\end{align*}
Thus, the lower right block of $\alpha_{d-t+1}^d$ is indeed $\alpha_{d-t}^{d-1}$ as desired.

If $\nonzero$ and $\lambda_t=0$, then $[\alpha_{d-t+1}]_{(f_{\lambda}, f_{\mu})}\not=0$ if and only if $\lambda=\mu$ by \Cref{equation1}. In this case \begin{equation*}
    [\alpha^d_{d-t+1}]_{(f_{\lambda}, f_{\mu})}=
    \begin{cases}
        (-1)^{d-0} \prod_{s=0}^{t-2} \binom{\sum_{j=1}^s (\mu_j+\mu_j)+\mu_{s+1}+2s}{\sum_{j=1}^s (\mu_j+\mu_j)+\mu_{s+1}+2s} =(-1)^{d} & \text{ if }\lambda=\mu,\\
         0 & \text{ if }\lambda\not=\mu.
    \end{cases}   
\end{equation*}
Moreover, notice that  there are $\binom{d-1}{t-2}$-many $(\mu_t,\dots,\mu_1)$'s in $ S_{t, d-t+1}$ such that $\mu_t=0$. It follows that the upper left corner is $(-1)^d\mathbbm{1}_{\binom{d-1}{t-2}}$ and the lower left corner is $\bigzero$ as desired.
\end{proof}

\section{Applications and Future Work}\label{applications}
This section explains the relationship between stable sheaf cohomology for $\PP^n({\bf k})$ and $\PP^n(\Z)$ and the homology groups of $C_{\bullet}(\underline{w})$ as well as a streamlined proof of \cite[Theorem 6.12]{raicukeller} via \Cref{main}.
Recall, as given in \cite{raicukeller}, the connection between $C_{\bullet}(\underline{w})$ and stable sheaf cohomology of ribbon and two column partition Schur functors applied to $\Omega$, the cotangent sheaf of projective space. The case $w_0 \geq 1$ corresponds to a ribbon, $\lambda/\mu$, which is a skew shape that contains no $2\times2$ squares and satisfies $\mu_{i} < \lambda_{i+1}$ whenever $\lambda_{i+1} > 0$. The columns of a ribbon yield a composition $\underline{w} = (w_0,\dots, w_d)$. The following figure is an example of a ribbon shape where $\mu = (3, 2, 2)$, $\lambda = (4,4,3,3)$ and $\underline{w} = (1, 1, 3, 2)$.

\begin{figure}[H]
    \centering
    \[\ydiagram{3+1, 2+2, 2+1, 3}\]
    \label{ribbon_shape}
\end{figure}

\noindent In \cite[Theorem 5.4]{raicukeller} it is proven that the stable cohomology of $\mathbb{S}_{\lambda/\mu}\Omega$ is the same as the homology of $C_\bullet^{\bf k} (\underline{w})$ with an appropriate shift. 

Additionally, in \cite[Theorem 5.7(2)]{raicukeller} it is shown for $\lambda$ a two column partition, i.e. the conjugate partition $\lambda' = (m, d)$ with integers $m,d\geq 0$, the stable cohomology groups of $\mathbb{S}_{\lambda}\Omega$ are given by the homology groups of $C_{\bullet}^{\bf k}(-m-d-1, 1^d)$ up to a shift. Notice the first weight is negative, whereas in the previous case all of the weights are positive. Further, they demonstrate the ranks of the homology groups of $C_{\bullet}^{\bf k}(-m-d-1, 1^d)$ and $C_{\bullet}^{\bf k}(m-d+1, 1^d)$ are the same via recursive formulas for the ranks and induction, motivating the conjecture that forms the basis of this paper. \Cref{main} gives a direct proof of this identification, and shows moreover that for fixed $d$ this identification is uniform with respect to these complexes.
\begin{theorem}[Raicu-VandeBogert]\label{sheaf}
Let $\lambda$ be a partition with two columns, i.e. the conjugate partition is of the form $\lambda' = (m, d)$ for some integers $m,d \geq 0$. Then for all $i$,
$$H_{st}^i(\mathbb{S}_{\lambda}\Omega) = H_{st}^{2m+1-i}(\mathbb{S}_{(d+1, 1^{m-d})}\Omega).$$
\end{theorem}
\begin{proof}
The composition associated to the hook $\mu = (d+1, 1^{m-d})$ is $w = (m-d+1, 1^d)$. If $x$ is evaluated to be $m - d+ 1$ in \Cref{main}, then
$$H_{st}^{i}(\mathbb{S}_{\lambda}\Omega) \stackrel{\text{\cite[Theorem 5.7(2)]{raicukeller}}}{=}H_{i-m}(C_{\bullet}(-m-d-1, 1^d)\otimes_R {\bf k})$$
\begin{equation*}
    \stackrel{\text{\Cref{main}}}{=} H_{i-m}(C_{\bullet}(m-d+1, 1^d)\otimes_R {\bf k}) \stackrel{\text{\cite[Theorem 5.4]{raicukeller}}}{=} H_{st}^{2m+1-i}(\mathbb{S}_{(d+1, 1^{m-d})}\Omega).\qedhere
\end{equation*}
\end{proof}

\Cref{main} also gives a uniform identification for projective space over the integers, which is the content of \Cref{sec}.

For the reader's convenience, we recall the statement of \Cref{sec} here.

\begin{theoremSec}
    Let $\Omega$ be the cotangent sheaf over $\mathbb{P}^{n}(\mathbb{Z})$ and $\mathbb{S}_{\lambda/ \mu}$ be a skew Schur functor. The following is true for all $i$.
    \begin{enumerate}
    \item[$i)$] $H^i(\mathbb{P}^n(\Z), \mathbb{S}_{\lambda/\mu} \Omega)$ is independent of $n$ for $n \gg 0$, and is denoted by\\ $H^i_{st}(\mathbb{S}_{\lambda/\mu} \Omega)$.
    \item[$ii)$] If $\mathbb{S}_{\lambda/\mu}$ is a ribbon Schur functor corresponding to  $(w_0, ... , w_d)$, then 
    \[H^i_{st}(\mathbb{S}_{\lambda/\mu} \Omega) = H_{|w|-i}(C_{\bullet}(\underline{w})\otimes_R \mathbb{Z}).\]
    \item[$iii)$] If $\lambda$ is a two column partition, i.e. the conjugate partition for $\lambda$ is of the form $\lambda' = (m, d)$ for integers $m,d \geq 0$, then 
    \[H_{st}^i(\mathbb{S}_{\lambda} \Omega) = H_{d+m-i}(\widecheck{C_{\bullet}}(-m-d-1,1^d)[-d]\otimes_R \mathbb{Z}),\]
    where $\widecheck{C_{\bullet}}(-m-d-1,1^d)$ denotes the dual of the complex $C_{\bullet}(-m-d-1,1^d)$.
    \item[$iv)$] Consider $\lambda$ as in $iii)$ and let $d \geq 1$, then
    \[H_{st}^i(\mathbb{S}_{\lambda} \Omega) = H_{st}^{2m+2-i}(\mathbb{S}_{(m-d+1, 1^d)}\Omega).\]
    \end{enumerate}
\end{theoremSec}

\begin{proof}
    By \cite{AB88}, skew Schur functors have universal resolutions by tensor products of exterior powers over arbitrary commutative rings. Therefore, $\mathbb{S}_{\lambda / \mu} \Omega$ admits a resolution which has terms given by direct sums of tensor products of exterior powers of $\Omega$. If $d = |\lambda|- |\mu| = (\sum \lambda_i) - (\sum \mu_i),$
    then these tensor powers of exterior powers of $\Omega$ will be of the form
    \[ \left (\exterior{d_1}\Omega \right )\otimes_{\mathscr{O}_{\mathbb{P}^n(\mathbb{Z})}} \dots \otimes_{\mathscr{O}_{\mathbb{P}^n(\mathbb{Z})}} \left(\exterior{d_h}\Omega\right) , \mbox{ for some } h \mbox{ and } \sum d_i = d.\]
    The cohomology of such sheaves is given by 
    \begin{align}\label{Hi}H^i\left (\mathbb{P}^n(\mathbb{Z}), \left (\exterior{d_1}\Omega \right )\otimes_{\mathscr{O}_{\mathbb{P}^n(\mathbb{Z})}} \dots \otimes_{\mathscr{O}_{\mathbb{P}^n(\mathbb{Z})}} \left(\exterior{d_h}\Omega\right)\right ) = \begin{cases}
        \mathbb{Z}, & i = d,\\
        0, & \mbox{otherwise,}
    \end{cases}\end{align}
    assuming that $n \geq d$. This calculation is followed by a double induction argument on the number of factors $h$ and the size of the last factor $d_h$ using the long exact sequence on cohomology and the short exact sequence
    \begin{align}\label{ses}0 \to \exterior{d} \Omega \to \exterior{d} (\mathbb{Z}^{n+1}) \otimes_{\mathscr{O}_{\mathbb{P}^n(\mathbb{Z})}} \mathscr{O}_{\mathbb{P}^n(\mathbb{Z})}(-d) \to \exterior{d-1} \Omega \to 0  \end{align}
    for $d \geq 1$. Also, necessary for this calculation is the fact that 
    \[H^j\left(\mathbb{P}^n(\mathbb{Z}),\left (\exterior{d_1}\Omega \right )\otimes_{\mathscr{O}_{\mathbb{P}^n(\mathbb{Z})}} \dots \otimes_{\mathscr{O}_{\mathbb{P}^n(\mathbb{Z})}} \left(\exterior{d_h}\Omega\right)(-i)\right) = 0\]
    for all $i$ in the range $1 \leq i \leq n-d$ and all $j$ (again under the assumption that $n \geq d$). This fact can similarly be proven using induction and \eqref{ses}.
    
    By \Cref{Hi}, applying the hypercohomology spectral sequence \cite[Proposition 5.7.9]{Weibel} to this resolution results in a single nonzero row on the first page. This nonzero row is then a chain complex of free abelian groups, which has homology groups matching the sheaf cohomology groups for $\mathbb{S}_{\lambda/\mu}$ up to a shift. Further, these chain complexes will be the same for $n\geq d$, so $H^i(\mathbb{P}^n(\mathbb{Z}), \mathbb{S}_{\lambda / \mu} \Omega)$ is independent of $n$ for $n \geq d$ proving $i)$. 
    
     In the special case when $\mathbb{S}_{\lambda / \mu}$ is a ribbon or two column Schur functor, the complexes of free abelian groups described above can be derived from explicit resolutions for the skew Schur functor. Such a resolution for ribbons is given by \cite[Theorem 5.1]{raicukeller}, and such a resolution for two column partitions is given by \cite[Section 4]{AB85} or \cite[Theorem 5.6]{raicukeller}. Finally, the differential maps can be derived by noting that \cite[Theorem 4.9 ii)]{raicukeller} generalizes to projective space over the integers, which states that the map induced on cohomology
     \[H^j(\exterior{a+b} \Omega) \to H^j(\exterior{a} \Omega) \otimes_{\mathbb{Z}} H^j(\exterior{b} \Omega)\]
     is multiplication by $\binom{a+b}{a}$. The statement of $ii)$ and $iii)$ is then obtained by taking homology, in analogy with the statement of \cite[Theorem 5.4]{raicukeller} and \cite[Theorem 5.7(2)]{raicukeller}.
    
    As $C_{\bullet}(-m-d-1, 1^d)\otimes_R \mathbb{Z}$ is a complex of free abelian groups, the universal coefficient theorem \cite{Hatcher} gives a relationship between its homology groups with the homology groups of its dual. Since $d >1$, the complex $C_{\bullet}(-m-d-1,1^d)\otimes_R \mathbb{Q}$ is exact, as the sheaf $\mathbb{S}_{\lambda}\Omega$ now considered over $\mathbb{Q}$ does not have nonzero sheaf cohomology groups by the Borel-Weil-Bott theorem. Thus, the homology of $C_{\bullet}(-m-d+1,1^d)\otimes_R \mathbb{Z}$ is only torsion. An application of the universal coefficient theorem yields
    \[H^i_{st}(\mathbb{S}_{\lambda} \Omega) = H_{d+m-i}(\widecheck{C_{\bullet}}(-m-d-1,1^d)[-d]\otimes_R \mathbb{Z})\] 
    \[= H_{i-m-1}(C_{\bullet}(-m-d-1,1^d)\otimes_R \mathbb{Z}).\]
    The rest of the proof of $iv)$ proceeds in analogy with our proof of \Cref{sheaf} of \cite{raicukeller} above now using $ii)$ and $iii)$ instead of \cite[Theorem 5.4]{raicukeller} and \cite[Theorem 5.7(2)]{raicukeller}.
\end{proof}

{\bf Acknowledgments}:
The heart of this paper was born during the PRAGMATIC school which took place in Catania, Italy, during the summer of 2023. The authors are grateful to the organizers for restarting this program. The authors would like to thank Claudiu Raicu for his interesting lectures and for introducing them to this conjecture. Further thanks go to Keller VandeBogert for patiently sharing his insights about the complexes, and to Alessio Sammartano for helpful discussions. The authors also thank the anonymous referee for thoroughly reading the paper and making valuable suggestions that improved both its exposition and the mathematical content.

\bibliographystyle{plain}
\bibliography{biblio}

\end{document}